\documentclass{amsart}
\usepackage{amsmath,amsthm,amssymb,amsfonts}

\usepackage[english]{babel}
\usepackage{makecell}
\usepackage{setspace}
\usepackage{array}
\usepackage{multirow}

\usepackage{graphicx}
\usepackage{tikz}

  \usepackage{epsfig}
\usepackage[pagebackref, colorlinks=true]{hyperref}

\hypersetup{urlcolor=blue, citecolor=blue}

\def\doi#1{   {\href{http://dx.doi.org/#1}
   {{\mdseries\ttfamily DOI}}}}

\usepackage{graphics}

\def\p{\partial}

\def\b{\beta}

\newcommand{\al}{\alpha}    
\newcommand{\de}{\delta}    \newcommand{\De}{\Delta}
  \newcommand{\ep}{\varepsilon}
  
    \newcommand{\la}{\lambda}

\newcommand{\ga}{\gamma}    
\newcommand{\R}{\mathbb{R}}
\newcommand{\N}{\mathbb{N}}
\newcommand{\ti}{\tilde }

\newcommand{\beeq}{\begin{equation}}\newcommand{\eneq}{\end{equation}}
\newcommand{\Rmnum}[1]{\uppercase \expandafter{\romannumeral #1}}

\newcommand{\supp}{\text{supp}}
\newcommand{\exponent}{e^{it(-\Delta)^{\frac{a}{2}}}}
\newcommand{\norm}[1]{\|#1\|}
\newcommand{\rate}{R_\de^a}

\newenvironment{prf}{\noindent {\bf Proof.} }{\endprf\par}
\def \endprf{\hfill  {\vrule height6pt width6pt depth0pt}\medskip}

\def\<{\langle}             \def\>{\rangle}
\def\({\left(}                 \def\){\right)}
\numberwithin{equation}{section}
\newtheorem{thm}{Theorem}[section]
 
 \newtheorem{lem}[thm]{Lemma}
 \newtheorem{prop}[thm]{Proposition}
 
 \newtheorem{rem}[thm]{Remark}

\title[sharp convergence rate]
{sharp convergence rate on Schr\"{o}dinger type operators}

\author{Meng Wang}
\address{School of Mathematical Sciences\\ Zhejiang University\\Hangzhou 310058, P. R. China}\email{mathdreamcn@zju.edu.cn }

\author{Shuijiang Zhao$^{*}$}\thanks{* Corresponding author}
\address{School of Mathematical Sciences\\ Zhejiang University\\Hangzhou 310058, P. R. China}\email{zhaoshuijiang@zju.edu.cn }



\begin{document}
\bibliographystyle{plain}

\maketitle

\begin{abstract}
For Schr\"{o}dinger type operators in one dimension, we consider the relationship between the convergence rate and the regularity for initial data. By establishing the associated frequency-localized maximal estimates, we prove sharp results up to the endpoints. The optimal range for the wave operator in all dimensions is also obtained.
\end{abstract}

\section{Introduction}
For $ a> 0 $, the solution of the fractional Schr\"{o}dinger equation
\begin{equation*}
	i\p_tu+(-\De)^{\frac{a}{2}}u=0,\quad u(x,0)=f(x)\in H^s(\R^n)
\end{equation*}
can be formally written as the Schr\"{o}dinger type operator
\begin{equation*}
S^af(x,t)=\exponent f(x)=\frac{1}{(2\pi)^n}\int_{\R^n}e^{i(x\xi+t|\xi|^a)}\hat{f}(\xi)d\xi.
\end{equation*}
Here, $ H^s(\R^n) $ is  the inhomogeneous Sobolev space of order $ s $. When $ a=2 $, $ S^2 $ is the standard Schr\"{o}dinger operator.
\par Carleson \cite{Carleson} initially proposed the problem of determining the optimal regularity $ s $ such that
\begin{equation}
	\lim_{t\to 0}S^af(x,t)=f(x),\quad  a.e.\quad x\in\R^n\label{point}
\end{equation}
holds for all $ f\in H^s(\R^n) $ when $ a=2 $. He proved that the almost everywhere pointwise convergence holds for  $ s\geq \frac{1}{4} $  when $ n=1 $ and $ a=2 $, which is also necessary as shown by Dahlberg-Kenig \cite{1981DK}. For the standard Schr\"{o}dinger operator $ S^2 $ in higher dimensions, this problem is much more complicated. Du-Guth-Li \cite{DuGuthLi} and Du-Zhang \cite{DuZhang}  showed that (\ref{point}) holds for $ s>\frac{n}{2n+2} $ when $ n\geq 2$ and $ a=2 $, which is sharp up to the endpoints according to the counterexamples by Bourgain \cite{Bourgain16} and Luc\`a-Rogers \cite{LucaRogers19}. See \cite{sjolin1987regularity, Vega1988, Bourgain1995, 2000TV2, lee2006, Bourgain2013} and references therein for more previous results related to this problem.
\par For the case $ a>1 $, Sj\"{o}lin \cite{sjolin1987regularity} showed that (\ref{point}) holds for all $ f\in H^s(\R) $ if and only if $ s\geq \frac{1}{4} $ when $ n=1 $. Miao-Yang-Zheng \cite{MiaoYangZheng2015} and Cho-Ko \cite{ChoKo2022} obtained partial results in higher dimensions when $ a>1 $. For the case $ 0<a<1 $, it was shown by Walther \cite{1995walther} that the critical regularity in one spatial dimension is $ \frac{a}{4} $. However, the endpoint problem $ s=\frac{a}{4} $ remains open. Related results in higher dimensions can be found in \cite{Zhangchunjie}. For the case $ a=1 $, it is well known that, in any dimension,  (\ref{point}) holds  if and only if $ s>\frac{1}{2} $. See \cite{1999Walther} and \cite{2023WZ} for more details.
\par A natural generalization of the pointwise convergence problem is to  consider the convergence rate of $ S^a f$ as $t\to 0 $. More specifically, we are interested in  the sharp range of $ \de $ such that
\begin{equation}
S^a f(x,t)-f(x)=o(|t|^\de),\quad  a.e.\quad x\in\R^n   \quad \text{as} \quad t\to 0.\label{convergence rate}
\end{equation}
holds whenever $ f\in H^s(\R^n) $. Since  critical values of regularity are unknown in higher dimensions for general $ a>0 $, we focus on  the case  $ n=1 $ in this paper. 
\par By a slight modification of \cite[Proposition 2.1]{CFW2018rate}, it is clear that $\de<1 $ is necessary if (\ref{convergence rate}) holds. To the best of the authors' knowledge, the earliest results regarding this problem were established by Cowling \cite{harmonicsemigroup,cowling1983pointwise}, using harmonic analysis on semigroups.

\begin{thm}[\cite{cowling1983pointwise}, p.85]\label{thm1Cowling}
Let $ H $ be a self-adjoint operator on $ L^2(\R^n) $, and let $ \de \in (\frac{1}{2},1) $. Suppose that $ f\in L^2(\R^n) $, and $ |H|^\de f\in L^2(\R^n) $. Let $ M_\de f $ be defined by the formula
$$  M_\de f(\cdot)=\sup\{|t|^{-\de}|e^{itH}f-f|:t\in \R\} .$$
Then $ M_\de f $ is in $ L^2(\R^n) $, and 
\begin{equation*}
\norm{M_\de f}_{L^2(\R^n)}\leq C(\de)\norm{ |H|^\de f}_{L^2(\R^n)}.
\end{equation*}
\end{thm}
Applying Theorem \ref{thm1Cowling} with $ H=(-\Delta)^{\frac{a}{2}} $, by a standard argument (see, e.g., \cite[Theorem 5]{sjolin1987regularity}), we obtain 
	\begin{equation}
	\lim_{t\to 0}\rate f(x,t)=0,\quad a.e.\quad x\in \R^n, \label{equivalence}
	\end{equation}
	holds for all $ f\in H^s(\R^n) $ when $ s=a\de $ and $ \frac{1}{2}<\de<1 $. Here and in what follows we use the notation
	\begin{equation*}
	\rate f(x,t)=\frac{\exponent f(x)-f(x)}{|t|^\de}. 
	\end{equation*}
	Since (\ref{equivalence}) is equivalent to (\ref{convergence rate}), we obtain the convergence rate $ \de $ for all $ f\in H^{a\de}(\R^n) $ in any dimension  when $ \frac{1}{2}<\de<1 $.
\par Cao-Fan-Wang \cite{CFW2018rate} obtained partial results for the case $ a>1 $ in one spatial dimension when $ \frac{1}{4}\leq s\leq \frac{a}{2} $.
\begin{thm}[\cite{CFW2018rate}, Theorem 1.2]\label{thm2CFM}
Let $ n=1, a>1, 0\leq \de <1 $,  and assume $ a\de+\frac{1}{4}\leq s $. Then (\ref{convergence rate}) holds for all $ f\in H^s(\R) $.
\end{thm}
 When $ n=1 $, we extend the above results to the sharp range for the case $ a>1 $.
 \begin{thm}\label{thm3}
 	Let $ n=1$, $a>1$, $0\leq \de < 1 $,  and assume 
 	\begin{equation}
 		(a-1)\de+\max\{\de,\frac{1}{4}\}<s.
 	\end{equation}
 	Then (\ref{convergence rate}) holds for all $ f\in H^s(\R) $. Conversely,  if (\ref{convergence rate}) holds for all $ f\in H^s(\R) $, then $ (a-1)\de+\max\{\de,\frac{1}{4}\}\leq s $.
 \end{thm}
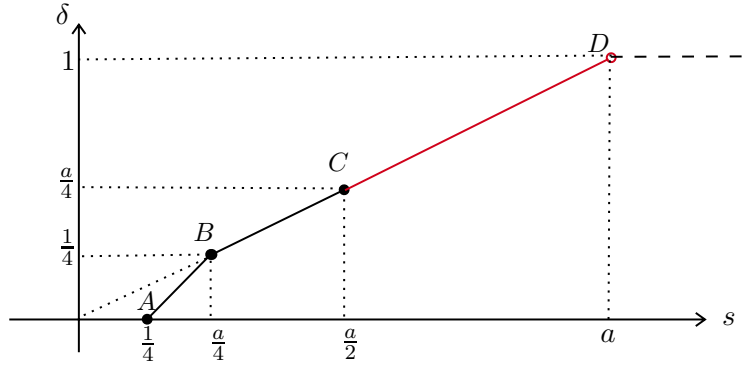
\begin{figure}
	\centering
	\tikzset{every picture/.style={line width=0.75pt}} 
	
	\begin{tikzpicture}[x=0.5pt,y=0.5pt,yscale=-1,xscale=1]
	
	\draw  [dash pattern={on 0.84pt off 2.51pt}]  (103.33,148.67) -- (301.33,149.67) ;
	\draw  (48,248.87) -- (574.33,248.87)(100.63,25.67) -- (100.63,273.67) (567.33,243.87) -- (574.33,248.87) -- (567.33,253.87) (95.63,32.67) -- (100.63,25.67) -- (105.63,32.67)  ;
	\draw    (152.33,248.67) -- (198.69,201.35) ;
	\draw [shift={(200.33,199.67)}, rotate = 314.41] [color={rgb, 255:red, 0; green, 0; blue, 0 }  ][line width=0.75]      (0, 0) circle [x radius= 3.35, y radius= 3.35]   ;
	\draw [shift={(152.33,248.67)}, rotate = 314.41] [color={rgb, 255:red, 0; green, 0; blue, 0 }  ][fill={rgb, 255:red, 0; green, 0; blue, 0 }  ][line width=0.75]      (0, 0) circle [x radius= 3.35, y radius= 3.35]   ;
	\draw    (201.33,199.67) -- (301.33,150.67) ;
	\draw [shift={(301.33,150.67)}, rotate = 333.9] [color={rgb, 255:red, 0; green, 0; blue, 0 }  ][fill={rgb, 255:red, 0; green, 0; blue, 0 }  ][line width=0.75]      (0, 0) circle [x radius= 3.35, y radius= 3.35]   ;
	\draw [shift={(201.33,199.67)}, rotate = 333.9] [color={rgb, 255:red, 0; green, 0; blue, 0 }  ][fill={rgb, 255:red, 0; green, 0; blue, 0 }  ][line width=0.75]      (0, 0) circle [x radius= 3.35, y radius= 3.35]   ;
	\draw  [dash pattern={on 0.84pt off 2.51pt}]  (104.33,245.67) -- (203.33,199.67) ;
	\draw [color={rgb, 255:red, 208; green, 2; blue, 27 }  ,draw opacity=1 ]   (302,151) -- (501.23,51.71) ;
	\draw [shift={(503.33,50.67)}, rotate = 333.51] [color={rgb, 255:red, 208; green, 2; blue, 27 }  ,draw opacity=1 ][line width=0.75]      (0, 0) circle [x radius= 3.35, y radius= 3.35]   ;
	\draw  [dash pattern={on 0.84pt off 2.51pt}]  (103,201) -- (199.33,199.67) ;
	\draw  [dash pattern={on 0.84pt off 2.51pt}]  (200.33,197.67) -- (200.33,247.67) ;
	\draw  [dash pattern={on 0.84pt off 2.51pt}]  (301.33,246.67) -- (301.33,151.67) ;
	\draw  [dash pattern={on 0.84pt off 2.51pt}]  (502.33,52.67) -- (501.33,249.67) ;
	\draw  [dash pattern={on 4.5pt off 4.5pt}]  (503,50) -- (604.33,49.67) ;
	\draw  [dash pattern={on 0.84pt off 2.51pt}]  (100,52) -- (504.33,49) ;

	\draw (585,241.4) node [anchor=north west][inner sep=0.75pt]  [font=\large]  {$s$};
	\draw (81,7.4) node [anchor=north west][inner sep=0.75pt]  [font=\large]  {$\delta $};
	\draw (85,43.4) node [anchor=north west][inner sep=0.75pt]  [font=\large]  {$1$};
	\draw (82,131.4) node [anchor=north west][inner sep=0.75pt]  [font=\large]  {$\frac{a}{4}$};
	\draw (83,178.4) node [anchor=north west][inner sep=0.75pt]  [font=\large]  {$\frac{1}{4}$};
	\draw (493,254.4) node [anchor=north west][inner sep=0.75pt]  [font=\large]  {$a$};
	\draw (142,225.4) node [anchor=north west][inner sep=0.75pt]    {$A$};
	\draw (185,173.4) node [anchor=north west][inner sep=0.75pt]    {$B$};
	\draw (287,120.4) node [anchor=north west][inner sep=0.75pt]    {$C$};
	\draw (483,30.4) node [anchor=north west][inner sep=0.75pt]    {$D$};
	\draw (144,251.4) node [anchor=north west][inner sep=0.75pt]  [font=\large]  {$\frac{1}{4}$};
	\draw (197,254.4) node [anchor=north west][inner sep=0.75pt]  [font=\large]  {$\frac{a}{4}$};
	\draw (297,253.4) node [anchor=north west][inner sep=0.75pt]  [font=\large]  {$\frac{a}{2}$};
	\end{tikzpicture}
	\caption{The relationship between the sharp convergence rate $ \de $ and the regularity $ s $ for $ a>1 $ when $ n=1 $. The  endpoint problem for the open edge $ (A,B) $ and the closed edge $ [B,C] $ remains open.}
	\label{1}
\end{figure}

We also obtain the following sharp results for the case $ 0<a<1 $.
\begin{thm}\label{thm4}
	Let $ n=1$, $0<a<1$, $0\leq \de < 1 $,  and assume $ \max\{\frac{1}{4}, \de\}\cdot a <s $.Then (\ref{convergence rate}) holds for all $ f\in H^s(\R) $. Conversely,  if (\ref{convergence rate}) holds for all $ f\in H^s(\R) $, then $  \max\{\frac{1}{4}, \de\}\cdot a\leq s $.
\end{thm}
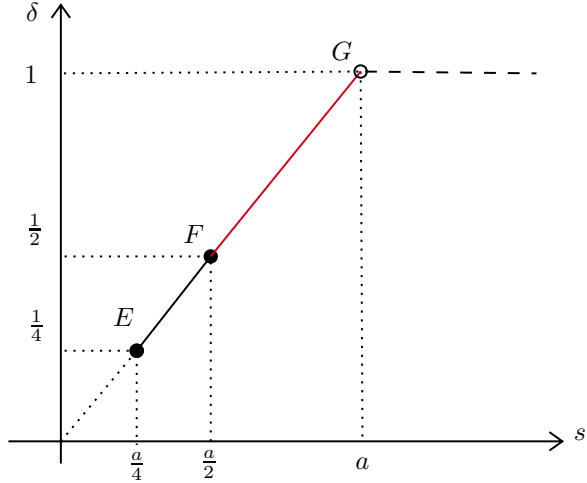
\begin{figure}
	\centering
	\tikzset{every picture/.style={line width=0.75pt}} 
	
	\begin{tikzpicture}[x=0.7pt,y=0.7pt,yscale=-1,xscale=1]
	
	\draw  (71,249) -- (370.33,249)(99.33,13) -- (99.33,261) (363.33,244) -- (370.33,249) -- (363.33,254) (94.33,20) -- (99.33,13) -- (104.33,20)  ;
	\draw  [dash pattern={on 0.84pt off 2.51pt}]  (180.33,249) -- (180.33,149) ;
	\draw  [dash pattern={on 0.84pt off 2.51pt}]  (140.33,202.35) -- (140.33,251) ;
	\draw [shift={(140.33,200)}, rotate = 90] [color={rgb, 255:red, 0; green, 0; blue, 0 }  ][line width=0.75]      (0, 0) circle [x radius= 3.35, y radius= 3.35]   ;
	\draw  [dash pattern={on 0.84pt off 2.51pt}]  (101.33,200) -- (140.33,200) ;
	\draw  [dash pattern={on 0.84pt off 2.51pt}]  (101.33,149) -- (180.33,149) ;
	\draw  [dash pattern={on 0.84pt off 2.51pt}]  (99.33,249) -- (140.33,200) ;
	\draw    (140.33,200) -- (180.33,149) ;
	\draw [shift={(180.33,149)}, rotate = 308.11] [color={rgb, 255:red, 0; green, 0; blue, 0 }  ][fill={rgb, 255:red, 0; green, 0; blue, 0 }  ][line width=0.75]      (0, 0) circle [x radius= 3.35, y radius= 3.35]   ;
	\draw [shift={(140.33,200)}, rotate = 308.11] [color={rgb, 255:red, 0; green, 0; blue, 0 }  ][fill={rgb, 255:red, 0; green, 0; blue, 0 }  ][line width=0.75]      (0, 0) circle [x radius= 3.35, y radius= 3.35]   ;
	\draw  [dash pattern={on 0.84pt off 2.51pt}]  (261.33,49) -- (262.33,250) ;
	\draw  [dash pattern={on 0.84pt off 2.51pt}]  (101,50) -- (261.33,49) ;
	\draw [color={rgb, 255:red, 208; green, 2; blue, 27 }  ,draw opacity=1 ]   (180.33,149) -- (261.33,49) ;
	\draw  [dash pattern={on 4.5pt off 4.5pt}]  (263.68,49.02) -- (356.33,50) ;
	\draw [shift={(261.33,49)}, rotate = 0.6] [color={rgb, 255:red, 0; green, 0; blue, 0 }  ][line width=0.75]      (0, 0) circle [x radius= 3.35, y radius= 3.35]   ;

	\draw (79,10.4) node [anchor=north west][inner sep=0.75pt]    {$\delta $};
	\draw (375,240.4) node [anchor=north west][inner sep=0.75pt]    {$s$};
	\draw (78,44.4) node [anchor=north west][inner sep=0.75pt]    {$1$};
	\draw (79,124.4) node [anchor=north west][inner sep=0.75pt]    {$\frac{1}{2}$};
	\draw (80,176.4) node [anchor=north west][inner sep=0.75pt]    {$\frac{1}{4}$};
	\draw (134,253.4) node [anchor=north west][inner sep=0.75pt]    {$\frac{a}{4}$};
	\draw (173,252.4) node [anchor=north west][inner sep=0.75pt]    {$\frac{a}{2}$};
	\draw (257,256.4) node [anchor=north west][inner sep=0.75pt]    {$a$};
	\draw (126,175.4) node [anchor=north west][inner sep=0.75pt]    {$E$};
	\draw (164,130.4) node [anchor=north west][inner sep=0.75pt]    {$F$};
	\draw (244,31.4) node [anchor=north west][inner sep=0.75pt]    {$G$};

	\end{tikzpicture}        
	\caption{The relationship between $ \de $ and $ s $ for the case $ 0<a<1 $ when $ n=1 $. The endpoint problem for the closed edge $ [E,F] $ is unknown. }
	\label{2}
\end{figure}

\begin{rem}
	For the case $ 0<a\neq 1 $ and $ \frac{1}{2}<\de<1 $, the borderline $ s=a\de $ is covered by Theorem \ref{thm1Cowling} (see the open edges $ (C,D) $ and $ (F,G) $ in Figure \ref{1} and Figure \ref{2} respectively). Note that, for the case $ 0<a< 1 $, there is a surprising jump for $ \de $ when $ s $ is bigger than the critical regularity $ \frac{a}{4} $ even though we do not know whether (\ref{point})  holds true for the endpoint $ s=\frac{a}{4} $.
\end{rem}
For the exceptional case $ a=1 $ which corresponds to the  wave operator, we can obtain the optimal range in all dimensions $ n\geq 1 $. As in the case $ 0<a<1 $, there is also a surprising jump for $ \de $ when $ s>\frac{1}{2} $.
\begin{thm}\label{thm5}
	Let $ n\geq 1$, $ a=1 $ and $0\leq \de<1  $. Then (\ref{convergence rate}) holds for all $ f\in H^s(\R^n) $ if and only if $ s>\frac{1}{2} $ and $ s\geq \de $.
\end{thm}

\subsubsection*{Notations}
Throughout this paper, we shall use $ c $ and $ C$ to denote positive constants, which may differ from line to line. By $A\lesssim B  $, we mean that there exists a constant $ C $    independent of the relevant parameters such that $ A\leq CB $ and similarly we use $ A\ll B $ to mean that $ A$ is far less than $ B $. We write $ A\sim B $ if $ A\lesssim B $ and $ B\lesssim  A $. We denote the annulus $ \{\xi\in R:|\xi|\sim R\}$ by $ A(R) $. $ \hat{f} $ and $ \check{f} $ denote the Fourier transform and the inverse Fourier transform respectively.
 \par Our paper is organized as follows. We prove sufficient conditions and necessary conditions for the case  $a>1$ in Section \ref{positive} and  Section \ref{negative} respectively, which provides typical methods for other cases. The proof of Theorem \ref{thm4} and Theorem \ref{thm5} is given in Section \ref{remaining}.

\section{positive results for the case $a>1$}\label{positive}
 By a similar standard argument, it suffices to show the following maximal estimate.
\begin{prop}\label{2.1}
	Let $a>1$, $0\leq \de < 1 $. Then for any $ f\in H^s(\R) $ with $ s> (a-1)\de+\max\{\de,\frac{1}{4}\}$, we have
	\begin{equation}
	\Big\|\sup_{0<t<1}|R_\de^af(\cdot,t)|\Big\|_{L^2(B)}\leq C\norm{f}_{H^s(\R)},\label{maximal estimate}
	\end{equation}
	where $  B $ denotes the unit ball in $ \R $ and the constant $ C $ is independent of $ f $.
\end{prop}
\begin{proof}
	By Littlewood-Paley decomposition,
	\begin{equation*}
		f=\sum_{k=0}^{+\infty}f_k,
	\end{equation*}
	where $ \hat{f_0} $ is supported in $ B(0,1) $ and $\hat{f_k}  $ is supported in $ A(2^k) $ for $ k\geq 1 $. We claim that
	\begin{equation}
		\Big\|\sup_{0<t<1}|R_\de^af_0(\cdot,t)|\Big\|_{L^2(B)}\leq C\norm{f}_{L^2(\R)}.\label{lowfrequency}
	\end{equation}
	Without loss of generality, we assume $ \de>\frac{1}{2} $. Otherwise, we can replace $ \de $ by some $ \de'>\frac{1}{2} $ since $ |\rate f(x,t)|\leq |R_{\de'}^a f(x,t)| $ when $ t\in (0,1) $. Applying Theorem \ref{thm1Cowling} with $ H=(-\De)^{\frac{a}{2}} $, we obtain that 
	\begin{equation*}
			\Big\|\sup_{t\in \R}|R_\de^af_0(\cdot,t)|\Big\|_{L^2(\R)}\leq C\norm{f_0}_{H^{a\de}(\R)}.
	\end{equation*}
	Note that $ \hat{f_0} $ is supported in $ B(0,1) $. Thus,  the $ H^{a\de}(\R) $ norm of $ f_0 $ is controlled by the $ L^2(\R) $ norm of $ f $, which proves (\ref{lowfrequency}). To get (\ref{maximal estimate}), by Minkowski's inequality and direct summation, it is sufficient to show the following frequency-localized estimate.
\end{proof}

\begin{prop}\label{freqlocalized}
	If $ a>1 $ and $ 0\leq \de<1 $, then for $ k\geq 1 $ and $ s=(a-1)\de+\max\{\de,\frac{1}{4}\} $, we have 
	\begin{equation}
		\Big\|\sup_{0<t<1}|R_\de^af(\cdot,t)|\Big\|_{L^2(B)}\leq Ck2^{ks} \norm{f}_{L^2(\R)}\label{localized},
	\end{equation} 
	whenever $ \supp \hat{f}\subset A(2^k) $. Moreover, the factor $  k $ can be removed when $ \de\neq \frac{1}{4}$.
\end{prop}
We recall two important maximal estimates which will be used in the proof of Proposition \ref{freqlocalized}.
\begin{lem}[\cite{1983KenigRuiz}, \cite{sjolin1987regularity}]\label{sharp local lemma}
	If $ n=1$ and $ a>1 $, then 
	\begin{equation}
		\Big\|\sup_{0<t<1}|S^af(\cdot,t)|\Big\|_{L^2(B)}\leq C\norm{f}_{H^{\frac{1}{4}}(\R)}.\label{local estimate}
	\end{equation}
\end{lem}
This sharp local estimate  was first  proven by Kenig and Ruiz \cite[Theorem 1]{1983KenigRuiz} when $ a=2 $ and later generalized to the case $ a>1 $ by Sj\"{o}lin \cite[Theorem 3]{sjolin1987regularity}. We also need a variant of the above inequality, which was established by Dimou and Seeger \cite{2020DimouSeeger}.
\begin{lem}[\cite{2020DimouSeeger}, Proposition 2.1]\label{dimou}
	If $ J\subset [0,1] $ is an interval and $ 0<a\neq 1 $, then 
	\begin{equation}
		\Big\|\sup_{t\in J}|S^af(\cdot,t)|\Big\|_{L^2(\R)}\leq C(1+|J|^{\frac{1}{4}}2^{\frac{ka}{4}})\norm{f}_{L^2(\R)}
	\end{equation}
	whenever $ \supp \hat{f}\subset A(2^k) $.
\end{lem}
Now let us turn to prove Proposition \ref{freqlocalized}. The key idea is  to decompose the interval $ (0,1) $ according to the localized frequency.
\begin{proof}[Proof of Proposition \ref{freqlocalized}]
	It is clear that
	\begin{equation*}
		\sup_{0<t<1}|R^a_\de f(x,t)|\leq I_1(x)+I_2(x)+I_3(x),
	\end{equation*}
	where
	\begin{align*}
		I_1(x)&=\sup_{0<t\leq 2^{-ak}}|R^a_\de f(x,t)|,\\
		I_2(x)&=\sup_{2^{-ak}<t<2^{(1-a)k}}|R^a_\de f(x,t)|,\\
		I_3(x)&=\sup_{2^{(1-a)k}\leq t<1}|R^a_\de f(x,t)|.\\
	\end{align*}
	We first give the upper bound for $ \norm{I_3}_{L^2(B)} $. Note that
	\begin{equation*}
	I_3(x)\leq 2^{(a-1)k\de}\big(\sup_{0<t<1}|S^a f(x,t)|+|f(x)|\big).
	\end{equation*}
	We invoke Lemma \ref{sharp local lemma} to obtain that
	\begin{align*}
		\norm{I_3}_{L^2(B)}&\leq  2^{(a-1)k\de} \Big(\Big\| \sup_{0<t<1}|S^a f(x,t)|\Big\|_{L^2(B)}+\norm{f}_{L^2(B)}\Big)\\
		                   &\lesssim  2^{(a-1)k\de}\norm{f}_{H^{\frac{1}{4}}(\R)}\\
		                   &\lesssim 2^{ks}\norm{f}_{L^2(\R)}.
	\end{align*}
	In order to estimate $ I_1(x) $, we use Taylor's expansion 
	\begin{equation*}
	e^{it|\xi|^a}-1=\sum_{j=1}^{+\infty}\frac{(it|\xi|^a)^j}{j!}.
	\end{equation*}
	Thus we can then write 
	\begin{align*}
		R_\de^a f(x,t)&=\frac{1}{2\pi}t^{-\de}\int_{\R}e^{ix\xi}(	e^{it|\xi|^a}-1)\hat{f}(\xi)d\xi\\
		              &=\frac{1}{2\pi}\sum_{j=1}^{+\infty}\frac{i^jt^{j-\de}}{j!}\int_{\R}e^{ix\xi}|\xi|^{aj}\hat{f}(\xi)d\xi.
	\end{align*}
	Since $ \de<1 $, by triangle inequality, we get
	\begin{equation*}
		I_1(x)=\sup_{0<t\leq 2^{-ak}}|R^a_\de f(x,t)|\leq \sum_{j=1}^{+\infty}\frac{2^{-ak(j-\de)}}{j!}|(-\De)^{\frac{ja}{2}}f(x)|.
	\end{equation*}
	It follows from Plancherel's theorem and the  Taylor's expansion of exponential function $ e^x $ that
	\begin{align*}
		\norm{I_1}_{L^2(\R)}&\leq \sum_{j=1}^{+\infty}\frac{2^{-ak(j-\de)}}{j!}(C2^k)^{ja}\norm{f}_{L^2(\R)}\\
		                    &=2^{ak\de}\sum_{j=1}^{+\infty}\frac{C^{ja}}{j!}\norm{f}_{L^2(\R)}\\
		                    &\leq2^{ak\de}\exp(C^a)\norm{f}_{L^2(\R)}\\
		                    &\lesssim 2^{ks}\norm{f}_{L^2(\R)}.
	\end{align*}
	For the term $I_2(x)$,  we use dyadic decomposition towards interval $ (2^{-ak},2^{(1-a)k}) $ so that 
	$$ I_2(x)\leq \sup_{\substack{l\in\N,\\(a-1)k\leq l\leq ak+1}}\sup_{t\in J_l}|R_\de^a f(x,t)|,$$
	where $ J_l=[2^{-l},2^{-l+1}] $. Using the natural embedding $ l^1\hookrightarrow l^\infty $ , we have 
	$$ I_2(x)\leq C\sum_{(a-1)k\leq l\leq ak+1}2^{l\de}\big(\sup_{t\in J_l}|S^a f(x,t)|+|f(x)|\big).$$
	Applying Lemma \ref{dimou} with $ J=J_l $, we derive that
	\begin{align*}
		\norm{I_2}_{L^2(\R)}&\lesssim\sum_{(a-1)k\leq l\leq ak+1}2^{l\de}\Big( \Big\| \sup_{t\in J_l}|S^a f(.,t)| \Big\|_{L^2(\R)}+\norm{f}_{L^2(\R)}\Big)\\
		&\lesssim \sum_{(a-1)k\leq l\leq ak+1}2^{l\de}2^{-\frac{l}{4}}2^{\frac{ka}{4}}\norm{f}_{L^2(\R)}\\
		&\lesssim\begin{cases}
		2^{ks}\norm{f}_{L^2(\R)},&\text{if} \quad\de\neq \frac{1}{4},\\
		k2^{ks}\norm{f}_{L^2(\R)},&\text{if} \quad\de=\frac{1}{4}.
		\end{cases}
	\end{align*}
Combining all the estimates for  $\norm{I_i}_{L^2(B)}$, we finish the proof of  (\ref{localized}).
	
\end{proof}

\section{negative results for the case $a>1$}\label{negative}
In order to prove necessary conditions, we make use of Stein's maximal principle \cite[Chapter \Rmnum{10}, \S 3.4]{stein1993harmonic}.
\begin{prop}\label{maximal principle}
	 Let $ a>0 $, $ 0\leq \de<1 $ and assume that (\ref{equivalence}) holds for any $ f\in H^s(\R) $. Let $ \{t_n\}_{n=1}^\infty \subset(0,1)$ be a sequence  with $ \lim_{n\to\infty}t_n=0 $. Then there is a constant $ C $, independent of $ f\in H^s(\R) $ and $ \al>0 $, such that
	 \begin{equation*}
	 	 meas(\{x\in[-1,1]:\sup_{n}|\rate f(x,t_n)|>\al\})\leq C\al^{-2}\norm{f}^2_{H^s(\R)}
	 \end{equation*}
	 holds for any $ \al>0 $. Here we denote the Lebesgue measure of   $ A\subset\R $  by meas(A).
\end{prop}
It is easy to verify that the above proposition satisfies the assumptions in Stein's maximal principle, so we omit the details of proof. However, since this maximal principle only  focuses on a countable family of linear operators, it seems plausible to replace the supremum over $ \{t_n\}_{n=1}^\infty $ by the supremum over $ t\in(0,1) $ or any other bounded interval.
\par With the help of the above proposition, we are in position to present an explicit counterexample, which is a variant of construction in  \cite{1981DK}.
\begin{proof}[Proof of negative results in Theorem \ref{thm3}]
	We set $ e(z)=e^{iz}$ for simplicity. Choose a nontrivial  smooth function $ \phi $ with compact support in $ [-1,1] $ such that $ 0\leq \phi\leq 1 $  in $ \R $ and $ \int \phi(\xi)d\xi=1$. We define $ f_\la^\ep $  with $ 0<\ep\leq 1 $ and $ \la $  sufficiently large by its Fourier transform
	$$\widehat{f_\la^\ep}(\xi)=\frac{1}{\la}\phi\Big(\frac{\xi-\la^{1+\ep}}{\la}\Big).$$ 
	It is clear that $  \hat{f_\la^\ep}(\xi)$ is supported in  the interval $ [\la^{1+\ep}-\la,\la^{1+\ep}+\la] $, which implies that
	\begin{equation}
		\norm{f_\la^\ep}_{H^s(\R)}\sim \la^{(1+\ep)s}\la^{-\frac{1}{2}}\label{Hs-norm}.
	\end{equation}
	\par Another key point is to choose a time sequence with an appropriate decay rate.  Take $ t_n=n^{-\ga} $ for all $ n\geq 1 $. For the case $a>1$, $ \ga=2(a-1) $ is enough to guarantee the sharp counterexample. 
	Change of variables gives that
	\begin{equation*}
		|S^af_\la^\ep(x,t_n)|=\frac{1}{2\pi}\Big|\int_{|\xi|\leq 1}e(\Phi_\la^\ep(\xi;x,t_n))\phi(\xi)d\xi\Big|,
	\end{equation*}
	where 
	$$\Phi_\la^\ep(\xi;x,t_n)=x(\la^{1+\ep}+\la\xi)+t_n\la^{a(1+\ep)}(1+\la^{-\ep}\xi)^a.$$
	
	It follows from  Taylor's expansion that 
	\begin{equation*}
		(1+\la^{-\ep}\xi)^a=1+a\la^{-\ep}\xi+\frac{a(a-1)}{2}\la^{-2\ep}\xi^2+O(\la^{-3\ep})
	\end{equation*}
	for all $ \xi\in\supp\phi $. Consequently, 
	\begin{equation}
	\Phi_\la^\ep(\xi;x,t_n)=\la^{1+\ep}x+\la^{a(1+\ep)}t_n+E_1(\xi;x,t_n,\la,\ep)+E_2(\xi;t_n,\la,\ep)+E_3(\xi;t_n,\la,\ep), \label{phase function}
	\end{equation}
	where 
	\begin{align*}
		E_1(\xi)&=\xi(\la x+at_n\la^{a(1+\ep)-\ep}),\\
		E_2(\xi)&=\frac{a(a-1)}{2}t_n\la^{a(1+\ep)-2\ep}\xi^2,\\
		E_3(\xi)&=t_nO(\la^{a(1+\ep)-3\ep}).
	\end{align*}
	Here we use the shorthand notation $ E_i(\xi), i=1,2,3,$ for convenience if it does not cause any confusion. Set
	\begin{equation*}
		I_{\la,\ep}:=[-c_0\la^{-(1-\ep)},-c_0\la^{-(1-\ep)}/2],
	\end{equation*}
where $c_0$ is a  small positive constant to be chosen later. We claim that for any $ x\in  I_{\la,\ep}$, there exists a $ n(x,\la)\in \N $ such that 
	\begin{equation}
		|S^af_\la^\ep(x,t_{n(x,\la)})|\geq \frac{1}{4\pi} \label{lowerbound}.
	\end{equation}
	Note that  by Fourier inverse transform, we have $ f(x)=e(\la^{1+\ep}x)\check{\phi}(\la x) $. Since $ \check{\phi} $ is a Schwartz function, we have $ |f(x)|\leq C_N\la^{-\ep N} $ for any $ x\in  I_{\la,\ep}$. Then, by triangle inequality,
	\begin{equation}
		|S^af_\la^\ep(x,t_{n(x,\la)})-f(x)|\geq \frac{1}{8\pi},\label{difference}
	\end{equation} 
	when $ \la  $ is sufficiently large.  Since the first two terms in (\ref{phase function}) are  independent of $ \xi $,  we obtain that
	\begin{align*}
		|S^af_\la^\ep(x,t_{n(x,\la)})|=&\frac{1}{2\pi}\Big|\int_{|\xi|\leq 1}e(E_1(\xi)+E_2(\xi)+E_3(\xi))\phi(\xi)d\xi\Big|\\
		                  \geq &\frac{1}{2\pi}\Big|\int \phi(\xi)d\xi\Big|-\frac{1}{2\pi}\Big|\int_{|\xi|\leq 1}(e(E_1(\xi)+E_2(\xi)+E_3(\xi))-1)\phi(\xi)d\xi\Big|\\
		                  \geq &\frac{1}{2\pi}-\frac{1}{2\pi}\max_{|\xi|\leq 1}|e(E_1(\xi)+E_2(\xi)+E_3(\xi))-1|.
	\end{align*}
	Hence, to get (\ref{lowerbound}), it suffices to show that the absolute value of phase functions $ E_i(\xi), i=1,2,3,$ is small for our choices of $ x $ and $ n(x,\la) $. For any $ x\in  I_{\la,\ep}$, we can find a unique  $ n(x,\la)\in \N $ such that 
	$$x\in (-at_{n(x,\la)}\la^{(a-1)(1+\ep)},-at_{n(x,\la)+1}\la^{(a-1)(1+\ep)}].$$
	Since $ |x| $ is small and $ a>1 $,  some simple computation yields that
	\begin{equation}
		c|x|\leq t_{n(x,\la)}\la^{(a-1)(1+\ep)}\leq C|x|\label{bound for t}.
	\end{equation} 
	Moreover, recall that $ t_n=n^{-\ga} $, we easily obtain that
	\begin{equation*}
		t_{n(x,\la)}-t_{n(x,\la)+1}\leq Ct_{n(x,\la)}^{\b},\quad \b=\frac{\ga+1}{\ga}=\frac{2a-1}{2(a-1)}.
	\end{equation*}
	 Then for any $ |\xi|\leq 1 $ and $ x\in  I_{\la,\ep} $
	\begin{align*}
		|E_1(\xi;x,t_{n(x,\la)},\la,\ep)|\leq&a\la^{a(1+\ep)-\ep}(t_{n(x,\la)}-t_{n(x,\la)+1})\\
		                      \leq&C(\la^{a(1+\ep)-\ep}t_{n(x,\la)})t_{n(x,\la)}^{1/\ga}\\
		                      \leq&C\la |x|t_{n(x,\la)}^{1/\ga}\\
		                      \leq& c_0 C\la^\ep t_{n(x,\la)}^{1/\ga}\\
		                      \leq&c_0 C \la^{\ep-(a-1)\frac{1+\ep}{\ga}-\frac{1-\ep}{\ga}}.
		\end{align*}
		Note that $ \ep-(a-1)\frac{1+\ep}{\ga}-\frac{1-\ep}{\ga}\leq-\frac{1-\ep}{2}\leq 0 $ since $ \ep\leq 1 $ and $ \ga=2(a-1) $. Hence $ E_1(\xi)$ is sufficiently small provided that $ c_0$ is small enough. Now we give the upper bound for $E_2(\xi) $. Using (\ref{bound for t}), we have
		\begin{align*}
			|E_2(\xi;t_{n(x,\la)},\la,\ep)|\leq &Ct_{n(x,\la)}\la^{a(1+\ep)-2\ep}\\
			                       \leq &C|x|\la^{1-\ep}\leq c_0C.
		\end{align*}
		Similarly, we also obtain
		 $$E_3(\xi;t_{n(x,\la)},\la,\ep)=O(\la^{-\ep}).$$
		Combining all the estimates for $ |E_i(\xi)| $ yields the desired lower bound (\ref{lowerbound}).
		\par Notice that (\ref{bound for t}) implies
		\begin{equation*}
			|t_{n(x,\la)}|\sim \la^{-a(1+\ep)+2\ep}
		\end{equation*}
		for any $ x\in  I_{\la,\ep}$. Then it follows from (\ref{difference}) that
		$$|\rate f_\la^\ep(x,t_{n(x,\la)})|\gtrsim \la^{[a(1+\ep)-2\ep]\de}, \quad \forall x\in I_{\la,\ep}.$$
	Combining the above estimate with (\ref{Hs-norm}), we have
		$$\la^{\frac{\ep-1}{2}}\lesssim\la^{-[a(1+\ep)-2\ep]\de} \la^{(1+\ep)s-\frac{1}{2}},$$
		 according to Proposition \ref{maximal principle}. This clearly implies 
		 \begin{equation}
		 	[a(1+\ep)-2\ep]\de \leq (1+\ep)s-\ep/2,\label{relation}
		 \end{equation}
		 for any $ 0<\ep\leq 1. $ Let $ \ep=1 $ and $\ep\to 0+  $ in (\ref{relation}) respectively, we get 
		 $$ (a-1)\de+\max\{\de,\frac{1}{4}\}\leq s,$$
		 which completes the proof.
\end{proof}

\section{proof of Theorem \ref{thm4} and Theorem \ref{thm5}}\label{remaining}
Now we turn to the case $ 0<a<1 $, whose proof is analogous to that of Theorem \ref{thm3}. We first prove the sufficient condition. As we did in the proof of Proposition \ref{2.1},  it suffices to show the corresponding frequency localized estimate.
\begin{prop}
	If $0< a<1 $ and $ 0\leq \de<1 $, then for $ k\geq 1 $ and $ s=\max\{\de,\frac{1}{4}\}\cdot a $, we have
	\begin{equation*}
	\Big\|\sup_{0<t<1}|R_\de^af(\cdot,t)|\Big\|_{L^2(B)}\leq Ck2^{ks} \norm{f}_{L^2(\R)},
	\end{equation*} 
whenever $\supp \hat{f}\subset A(2^k) $.  Moreover, the factor $  k $ can be removed when $ \de\neq \frac{1}{4}$.
\end{prop}
\begin{proof}
	 We can also write
	\begin{equation*}
	\sup_{0<t<1}|R^a_\de f(x,t)|\leq \ti{I_1}(x)+\ti{I_2}(x),
	\end{equation*}
	where
	\begin{align*}
	\ti{I_1}(x)&=\sup_{0<t\leq 2^{-ak}}|R^a_\de f(x,t)|,\\
	\ti{I_2}(x)&=\sup_{2^{-ak}<t<1}|R^a_\de f(x,t)|.
	\end{align*}
	Proceeding as in the estimate for $ I_1(x) $ in Proposition \ref{freqlocalized},  we obtain 
	$$\norm{I_1}_{L^2(\R)}\leq C2^{ka\de}\norm{f}_{L^2(\R)}.$$
	For the second term $ I_2(x) $,  we use dyadic decomposition towards interval $ (2^{-ak},1) $ so that 
	$$ I_2(x)\leq \sup_{\substack{l\in\N,\\1\leq l\leq ak+1}}\sup_{t\in J_l}|R_\de^a f(x,t)|,$$
	where $ J_l=[2^{-l},2^{-l+1}] $. Again using Lemma \ref{dimou} with $ J=J_l $, we have
	\begin{align*}
	\norm{I_2}_{L^2(\R)}&\lesssim\sum_{1\leq l\leq ak+1}2^{l\de}\Big( \Big\| \sup_{t\in J_l}|S^a f(.,t)| \Big\|_{L^2(\R)}+\norm{f}_{L^2(\R)}\Big)\\
	&\lesssim \sum_{1\leq l\leq ak+1}2^{l\de}2^{-\frac{l}{4}}2^{\frac{ka}{4}}\norm{f}_{L^2(\R)}\\
	&\lesssim \begin{cases}
	2^{ks}\norm{f}_{L^2(\R)},&\text{if} \quad\de\neq \frac{1}{4},\\
	k2^{ks}\norm{f}_{L^2(\R)},&\text{if} \quad\de=\frac{1}{4}.
	\end{cases}
	\end{align*}
	This completes the proof.
\end{proof}
Now we consider the necessary condition for the case $ 0<a<1 $. The counterexample here is  quite similar to that in Theorem \ref{thm3}, except for a slight modification  of the time sequence.
\begin{proof}[Proof of negative results in Theorem \ref{thm4}]
	We use the same constructions $ \phi $ and $ f_\la^\ep $ utilized in the preceding section. Then $ \norm{f_\la^\ep}_{H^s(\R)}\sim \la^{(1+\ep)s}\la^{-\frac{1}{2}}$. 
	\par Take $ t_n=n^{-\ga} $. Here $ \ep$ and $ \ga $ satisfy  $ \ep(\ga+2-a)<a $. Arguing as before, we get that, for all $ x\in I_{\la,\ep}:=[-c_0\la^{-(1-\ep)},-c_0\la^{-(1-\ep)}/2] $, there exists a unique $ n(x,\la) $ such that
	$$x\in (-at_{n(x,\la)}\la^{(a-1)(1+\ep)},-at_{n(x,\la)+1}\la^{(a-1)(1+\ep)}].$$
  Moreover, we obtain 
\begin{equation*}
t_{n(x,\la)}\sim |x|\la^{(1-a)(1+\ep)}	\sim\la^{\ep(2-a)-a}\ll 1,
\end{equation*}
and
$$	t_{n(x,\la)}-t_{n(x,\la)+1}\leq Ct_{n(x,\la)}^{\frac{\ga+1}{\ga}}.$$
Using the identical notations as before, we have
$$|S^af_\la^\ep(x,t_{n(x,\la)})|\geq \frac{1}{2\pi}-\frac{1}{2\pi}\max_{|\xi|\leq 1}|e(E_1(\xi)+E_2(\xi)+E_3(\xi))-1|.$$
It is clear that
$$E_3(\xi;t_{n(x,\la)},\la,\ep)=O(\la^{-\ep}).$$
By a  simple computation, we get 
$$|E_1(\xi;x,t_{n(x,\la)},\la,\ep)|\leq C\la^{a\frac{\ep(\ga+2-a)-a}{\ga}}\ll 1,$$
and 
$$|E_2(\xi;t_{n(x,\la)},\la,\ep)|\leq Ct_{n(x,\la)}\la^{a(1+\ep)-2\ep}\leq c_0 C.$$
Proceeding as in the proof for the case $a>1 $, we have 
\begin{equation*}
	[a(1+\ep)-2\ep]\de \leq (1+\ep)s-\ep/2,
\end{equation*}
for $ 0<\ep<\frac{a}{\ga+2-a} $. Then we get $ a\de\leq s $ if $ \ep\to 0 $. Furthermore, let $ \ep\to\frac{a}{\ga+2-a} $ and $ \ga\to 0+$,  we obtain another necessary condition $ s\geq \frac{a}{4} $ .
\end{proof}

Now we turn to the case $ a=1 $. The positive result of Theorem \ref{thm5} follows from Theorem \ref{thm1Cowling}. It is easy to see that (\ref{convergence rate}) for the case $ a=1 $  implies
	\begin{equation}
		\lim_{t\to 0}S^1f(x,t)=f(x),\quad  a.e.\quad x\in\R^n.\label{pointwsie convergence}
	\end{equation}
So the failure of almost everywhere pointwise convergence for $ S^1f(x,t) $ yields the failure of (\ref{convergence rate}) when $ a=1 $. The counterexample for (\ref{pointwsie convergence}) when $ f\in H^{\frac{1}{2}}(\R^n) $ was established by \cite[Theorem 14.2]{1999Walther}, \cite[Lemma A.2]{ham2022dimension} and \cite[Section 3.1]{2023WZ}. Thus, it suffices to show the remaining necessary condition $ s\leq \de $. We point out that the following construction of radial function can be used to generalize the negative results of Theorem \ref{thm3} and Theorem \ref{thm4} to all dimensions $ n\geq 1 $.
\begin{proof}
	We only consider the case $ n\geq 2 $ since the case  $ n=1 $ is much easier. Choose the same smooth function $ \phi $ as above. Then define the Schwartz function $ g^\ep_\la $ with $ 0<\ep<\frac{1}{2} $ and $ \la $  sufficiently large by its Fourier transform
	$$\widehat{g_\la^\ep}(\xi)=\la^{-\frac{n+1}{2}}\phi\Big(\frac{|\xi|-\la^{1+\ep}}{\la}\Big).$$
	By Plancherel's theorem, it is easy to see  
	\begin{equation}
		\norm{g_\la^\ep}_{H^s(\R^n)}\sim \la^{(1+\ep)s}\la^{\ep\frac{n-1}{2}-\frac{1}{2}}.\label{n dimension norm}
	\end{equation}
	Now we take the time sequence to be $ t_n=n^{-2} $. It remains to show that for all $ x\in A(\la^{-1+\ep}) $, there exists a $ n(x)\in \N $ such that 
	\begin{equation}
		|R_\de^1g^\ep_\la(x,t_{n(x)})|\geq C|x|^{-(\frac{n-1}{2}+\de)}\la^{\ep\frac{n-1}{2}}.\label{goal}
	\end{equation}
	If (\ref{goal}) holds, it follows from Proposition \ref{maximal principle} and (\ref{n dimension norm}) that 
	\begin{equation*}
		\la^{-(1-\ep)(\frac{1}{2}-\de)}\la^{\ep\frac{n-1}{2}}\lesssim \la^{\ep\frac{n-1}{2}-\frac{1}{2}}\la^{(1+\ep)s}.
	\end{equation*}
	By a simple computation, we have 
	\begin{equation*}
	(1-\ep)\de\leq (1+\ep)s-\frac{\ep}{2}.
	\end{equation*}
	Let $ \ep\to 0+ $, we get the desired upper bound $ \de\leq s $.
	\par To get (\ref{goal}), we recall the well known asymptotic expansion (see, e.g., \cite[p. 580]{grafakos2008classical})
	\begin{equation*}
		\int_{\mathbb{S}^{d-1}}e^{ix\cdot \theta}d\theta=c_{\pm}|x|^{-\frac{n-1}{2}}e^{\pm i|x|}+O(|x|^{-\frac{n+1}{2}}) , \quad  \text{as } |x|\to\infty.
	\end{equation*}
	Then 
	\begin{align*}
		g_\la^\ep(x)=&\frac{1}{(2\pi)^n}\la^{-\frac{n+1}{2}}\int_{\R^n}e(x\cdot\xi)\phi\Big(\frac{|\xi|-\la^{1+\ep}}{\la}\Big)d\xi\\
		            =&\frac{1}{(2\pi)^n}\la^{-\frac{n+1}{2}}\int_{0}^{+\infty}\phi\Big(\frac{r-\la^{1+\ep}}{\la}\Big)\Big(\int_{\mathbb{S}^{d-1}}e(x\cdot r\theta)d\theta\Big)r^{n-1}dr\\
		            =&\frac{c_{\pm}}{(2\pi)^n}\la^{-\frac{n+1}{2}}|x|^{-\frac{n-1}{2}}\int_{0}^{+\infty}r^{\frac{n-1}{2}}\phi\Big(\frac{r-\la^{1+\ep}}{\la}\Big)e(\pm|x|r)dr\\
		            &+\frac{1}{(2\pi)^n}\la^{-\frac{n+1}{2}}\int_{0}^{+\infty}r^{n-1}\phi\Big(\frac{r-\la^{1+\ep}}{\la}\Big)O((r|x|)^{-\frac{n+1}{2}})dr\\
		            =&\frac{c_{\pm}}{(2\pi)^n}|x|^{-\frac{n-1}{2}}e(\pm |x|\la^{1+\ep})\int_{-1}^{1}(r+\la^\ep)^{\frac{n-1}{2}}\phi(r)e(\pm\la|x|r)dr\\
		            &+\frac{1}{(2\pi)^n}\la^{\frac{n-1}{2}}\int_{-1}^{1}(r+\la^\ep)^{\frac{n-1}{2}}\phi(r)O((\la^{1+\ep}|x|)^{-\frac{n+1}{2}})dr\\
		            =&I_{+}+I_{-}+\la^{\ep\frac{n-1}{2}}|x|^{-\frac{n-1}{2}}O((\la^{1+\ep}|x|)^{-1}).
	\end{align*}
Integration by part yields that 
\begin{equation*}
	I_{\pm}=\la^{\ep\frac{n-1}{2}}|x|^{-\frac{n-1}{2}}O((\la|x|)^{-1}).
\end{equation*}
Thus, we obtain, as long as $ x\in A(\la^{-1+\ep}) $, that 
\begin{equation}
	|g^\ep_\la(x)|\ll \la^{\ep\frac{n-1}{2}}|x|^{-\frac{n-1}{2}} \label{bound for g}.
\end{equation}
Arguing as in the previous section, there exists a unique $ n(x)\in \N $ such that 
\begin{equation*}
	|x|\in [t_{n(x)},t_{n(x)+1})
\end{equation*}
for all $ x\in A(\la^{-1+\ep}) $. Moreover, we have 
\begin{equation}
	t_{n(x)}\sim |x|,\label{bound for tn in section4}
\end{equation}
and 
\begin{equation}
	t_{n(x)}-t_{n(x)+1}\leq C|x|^2.\label{error for t}
\end{equation}
Proceeding as in the estimate for $g_\la^\ep(x)$, we can write 
\begin{equation}
	S_\de^1g^\ep_\la(x,t)=\tilde{I}_{+}+\tilde{I}_{-}+\la^{\ep\frac{n-1}{2}}|x|^{-\frac{n-1}{2}}O((\la^{1+\ep}|x|)^{-1}),\label{expression for S1}
\end{equation}
where 
$$\tilde{I}_{\pm}=\frac{c_{\pm}}{(2\pi)^n}|x|^{-\frac{n-1}{2}}e((t\pm |x|)\la^{1+\ep})\int_{-1}^{1}(r+\la^\ep)^{\frac{n-1}{2}}\phi(r)e(\la(t\pm|x|)r)dr.$$
It follows from integration by part that 
\begin{equation}
	|\tilde{I}_{+}| \leq C\la^{\ep\frac{n-1}{2}}|x|^{-\frac{n-1}{2}}(\la|x|)^{-1}.\label{esitimate for I+}
\end{equation}
Now we estimate the main term $ \tilde{I}_{-} $. It is clear that 
\begin{equation*}
	|\la(t_{n(x)}-|x|)r|\leq \la(t_{n(x)}-t_{n(x)+1}) \lesssim  \la |x|^2 \lesssim\la^{-1+2\ep}\ll 1,        	                                                    
\end{equation*}
 for all $ r\in \supp\phi $ and $ x\in A(\la^{-1+\ep}) $. Applying triangle inequality as before, we have 
 \begin{equation}
 	|\tilde{I}_{-}|\gtrsim \la^{\ep\frac{n-1}{2}}|x|^{-\frac{n-1}{2}}.\label{estimate for I-}
 \end{equation}
 It is concluded from (\ref{expression for S1}), (\ref{esitimate for I+}) and (\ref{estimate for I-}) that
 \begin{equation}
 	|S_\de^1g^\ep_\la(x,t_{n(x)})|\gtrsim  \la^{\ep\frac{n-1}{2}}|x|^{-\frac{n-1}{2}}\label{lowerbound for S1}.
 \end{equation}
Finally, we prove (\ref{goal}) using (\ref{bound for tn in section4}), (\ref{bound for g}) and (\ref{lowerbound for S1}), which completes the proof.
\end{proof}

\begin{center}
ACKNOWLEDGMENTS
\end{center}

The first author is supported in part by NSFC12371100 and NSFC12171424. The authors would like to thank Professor Dashan Fan for his helpful discussion.
\vskip10pt


\end{document}